\documentclass[11pt]{amsart}
\usepackage{graphicx}
\usepackage{amssymb}
\usepackage{amsmath}
\usepackage{epstopdf}
\usepackage{amsthm}

\usepackage{fourier}
\usepackage{a4wide}

\newtheorem{definition}{Definition} 
\newtheorem{proposition}[definition]{Proposition} 
\newtheorem{theorem}[definition]{Theorem} 
\newtheorem{lemma}[definition]{Lemma} 
\newtheorem{corollary}[definition]{Corollary}

\newtheorem{remark}[definition]{Remark}
\newtheorem{wedgelemma}[definition]{Wedge Lemma} 

\newtheorem{example}[definition]{Example}

\newcommand{\rar}{\rightarrow}
\newcommand{\RR}{\mathbb{R}}
\newcommand{\NN}{\mathbb{N}}
\newcommand{\ZZ}{\mathbb{Z}}

\renewcommand{\dim}{\mathsf{dim}\ }

\newcommand{\lk}{\mathsf{lk}}
\newcommand{\sd}{\mathsf{sd}}
\newcommand{\Pa}{\mathrm{P}}
\newcommand{\dotcup}{\ensuremath{\mathaccent\cdot\cup}}
\newcommand{\Hom}{\mathrm{ Hom}}

\def\A{\mathcal{A}}

\def\C{\mathcal{C}}
\def\D{\Delta}
\def\D'{\Delta}

\def\H{\mathcal{H}}

\def\N{\mathbb{N}}

\def\U{\mathcal{U}}

\title{Hypergraph Coloring Complexes}
\author{Felix Breuer, Aaron Dall, Martina Kubitzke}
\address{F. B.: Fachbereich Mathematik und Informatik, Freie Universit\"at Berlin, Arnimallee 3, D-14195 Berlin}
\address{A. D.: Departament de Matem\`{a}tica Aplicada II, Universitat Polit\`{e}cnica de Catalunya, 
Jordi Girona 1-3, E-08034 Barcelona}
\address{M. K.: Fakult\"at f\"ur Mathematik, Universit\"at Wien, Garnisongasse 3, A-1090 Wien}
\keywords{hypergraph, coloring complex, chromatic polynomial, Ehrhart theory, Cohen-Macaulay, wedge lemma}
\subjclass[2000]{05C65, 05C15, 05E45, 52B05}
\begin{document}

\begin{abstract}

The aim of this paper is to generalize the notion of the coloring complex of a graph to hypergraphs. We present 
three different interpretations of those complexes -- a purely combinatorial one and two geometric ones. It is shown, 
that most of the properties, which are known to be true for coloring complexes of 
graphs, break down in this more general setting, e.g., Cohen-Macaulayness and partitionabilty. Nevertheless, we 
are able to provide bounds for the $f$- and $h$-vectors of those complexes which yield new bounds on chromatic polynomials 
of hypergraphs. Moreover, though it is proven that the coloring complex of a hypergraph has a wedge decomposition, we provide an example showing that in general this decomposition is not homotopy equivalent to a wedge of spheres. In addition, we can completely characterize those hypergraphs whose coloring complex is connected.
\end{abstract}

\maketitle

\section{Introduction}

Graph colorings have been studied intensively since the mid-nineteenth century. 
One common approach to solving problems regarding either the chromatic number or the chromatic polynomial of a graph is to transfer the graph theoretic problem into the languages of topology and algebraic combinatorics.
For example, given a graph $G$ one can construct several simplicial complexes that give information about the chromatic number of $G$;
these include the neighborhood complex \cite{Lovasz,XieLiu}, the $\Hom$ complex \cite{BabsonKozlov2006,BabsonKozlov2007,Lovasz}, 
and the coloring complex of $G$ \cite{Steingrimsson}.
 Moreover, the coloring complex of $G$ encodes the chromatic polynomial $\chi_G(k)$  of $G$ (up to a shift) as the Hilbert polynomial of its Stanley-Reisner ideal. In particular, $$\frac{1}{z} \sum_{k\geq 0} (k^n - \chi_G(k)) z^k =\frac{h_0+h_1z+\ldots+h_{n-2}z^{n-2}}{(1-z)^n},$$ where $n$ is the number of vertices of $G$ and $h_0,\ldots,h_{n-2}$ is the $h$-vector of the coloring complex of $G$, see \cite[Theorem 13]{Steingrimsson}.
 A good deal of research has gone into the study of the topology of these complexes (see \cite{HershSwartzCol,Hultman,Jonsson}) which has led to bounds on the coefficients of the chromatic polynomial of $G$. In \cite{Hanlon} the coefficients of the chromatic polynomial were interpreted by means of the Hodge decomposition of the unique nontrivial homology group of the coloring complex.

 In this paper, we consider a generalization of coloring complexes to hypergraphs, originally introduced in \cite{LongRundell}. 
We first construct these hypergraph coloring complexes as abstract simplicial complexes via their combinatorics.
Then we realize them geometrically in two ways via (1) hyperplane arrangement decompositions of the sphere and (2) inside-out polytopes and Ehrhart theory and show that these constructions all yield the same simplicial complex, $\Delta_{H}$. This complex has the same relationship to the chromatic polynomial $\chi_H$ of a hypergraph $H$ as in the case of graphs: If $n$ is the number of vertices of $H$ and the maximal cardinality of an edge of $H$ is $m$, then $\Delta_{H}$ is a simplical complex of dimension $n-m-1$ with $h$-vector $(h_0,\ldots,h_{n-m})$ and we have
\begin{eqnarray}
\frac{1}{z} \sum_{k\geq 0} (k^n - \chi_H(k)) z^k =\frac{h_0+h_1z+\ldots+h_{n-m}z^{n-m}}{(1-z)^n}.
\label{eqn:chi-and-h}
\end{eqnarray}

As it turns out, coloring complexes of hypergraphs are much more intricate than coloring complexes of ordinary graphs. We show that most of the properties, which are known to be true for coloring complexes of graphs, break down in this more general setting. In general, hypergraph coloring complexes are neither pure nor connected, they are not Cohen-Macaulay, they are not partitionable and they do not have a non-negative $h$-vector. We also obtain some positive results, for example, we give bounds on the $f$-vectors of coloring complexes, which yield bounds on the chromatic polynomial of a hypergraph, and we characterize when hypergraph coloring complexes are connected. Finally, we provide an example of a hypergraph coloring complex that --~though being connected~-- is not homotopy equivalent to a wedge of spheres.

The paper is organized as follows. 
In Section 2 we provide the necessary background on simplicial complexes, hypergraphs, and Ehrhart theory. In Section \ref{subsect:combCol} we begin with a purely combinatorial definition of the hypergraph coloring complex, $\Delta_{H}$, of a hypergraph $H$. 
After giving some illuminating examples and fixing notation, we give an interpretation of $\Delta_{H}$ in terms of subspace arrangements (Theorem \ref{SubspaceArrangementInterpretation}) that is a generalization of the hyperplane arrangement interpretation of the coloring complex of a graph given in \cite{HershSwartzCol}.
Using this interpretation, we prove in Theorem \ref{prop:nonCM} that hypergraph coloring complexes are not, in general, Cohen-Macaulay.

In Section \ref{sec:ehrhart}, we give a third interpretation of the coloring complex in terms of inside-out polytopes and Ehrhart theory. 
After a brief review of $f$- and $h$-vectors of polytopal complexes and polynomials, we compute the $f$- and $h$-vectors of certain subcomplexes of the coloring complex and apply these results to give upper and lower bounds on the $f$-vector of the chromatic polynomial in theorems in Theorems~\ref{thm:subgraph} and \ref{thm:f-vector} and Corollary \ref{cor:f-vector}. We conclude this section by observing that the $h$-vector of the coloring complex may have negative entries.

Finally, in Section \ref{sec:homtype}, we analyze the homotopy type of the coloring complex of a hypergraph. 
As a consequence of the Wedge Lemma \cite[Wedge Lemma 6.1]{HRW-Koszul}, we obtain a wedge decomposition of $\Delta_{H}$ in Proposition \ref{prop:homotopy}. Unfortunately, the complexes appearing in this decomposition are not single spheres, but they are joins of spheres with certain order complexes which heavily depend on the structure of the underlying hypergraph. Even though those order complexes can be associated to smaller hypergraphs again, it is not clear, in general, what they look like.  We then give a characterization of connectedness of $\Delta_{H}$ in terms of the underlying hypergraph (see Proposition \ref{prop:connectedness}). We conclude the article by constructing a hypergraph coloring complex that does not have the homotopy type of a wedge of spheres.	

\section{Basic definitions and preliminaries}
In this section, we provide the basic definitions and facts which will be needed for the understanding of this paper. More specific notions and 
results which are only important in special places are stated within the corresponding section.

\subsection{Simplicial Complexes} \label{subsect:simplicial}

For a positive integer $n\in \N$ we use $[n]$ to denote the set $\{1,\ldots,n\}$. 
An \emph{(abstract) simplicial complex} $\Delta$ on vertex set $[n]$ is a collection of subsets of $[n]$ such that $\emptyset\in \Delta$ 
and if $F\in \Delta$ and $G\subsetneq F$, then also $G\in \Delta$. The elements of $\Delta$ are called \emph{faces} of $\Delta$. 
Faces which are singletons and inclusion wise maximal faces are referred to as \emph{vertices} and \emph{facets} of $\Delta$, 
respectively. 
The \emph{dimension} of a face $F\in \Delta$ equals its cardinality minus $1$ and the \emph{dimension} $\dim\Delta$ of $\Delta$ is the maximum 
dimension of its faces. If all facets of $\Delta$ are of the same dimension, then we call the simplicial complex \emph{pure}. 
The information about the numbers of faces of a certain dimension of a $(d-1)$-dimensional simplicial complex $\Delta$ is collected 
in its \emph{$f$-vector} $f(\Delta)=(f_{-1}(\Delta),f_0(\Delta),\ldots, f_{d-1}(\Delta))$, where
\begin{equation*}
 f_i(\Delta)=\#\{F\in \Delta~:~\dim F=i\}
\end{equation*}
for $-1\leq i\leq d-1$. For several purposes, it is more convenient to consider the so-called \emph{$h$-vector} of $\Delta$ which is 
the vector $h(\Delta)=(h_0(\Delta),\ldots,h_d(\Delta))$ determined by the relation
\begin{equation} \label{eq:ftoh}
 \sum_{i=0}^{d} h_i(\Delta) t^{d-i } =\sum_{i=0}^{d} f_{i-1}(\Delta) (t-1)^{d-i}.
\end{equation}

For a face $F  \in \Delta$, we write $\lk_\Delta(F) = \{ G \in \Delta~:~F \cap G = \emptyset,~F \cup G \in \Delta \}$ 
for the \emph{link} of $F$ in $\Delta$ and we denote by $\partial F$ the simplicial complex of all $G \subsetneq F$ 
that lie in the \emph{boundary} of the simplex $F$. We use $\Delta^n$ to denote the standard $(n-1)$-simplex, i.e., $\Delta^n=2^{[n]}$. 
Given two simplicial complexes $\Delta$ and $\Gamma$ the \emph{join} of $\Delta$ and $\Gamma$ is the simplicial complex given as 
$\Delta\ast\Gamma=\{F\cup G~:~F\in\Delta,~G\in\Gamma\}$. 
The \emph{barycentric subdivision} $\sd(\Delta)$ of $\Delta$ is the simplicial complex on vertex set 
$\Delta \setminus \{ \emptyset\}$ whose simplices
are flags $A_0 \subsetneq A_1 \subsetneq \cdots \subsetneq A_i$ of 
elements $A_j \in \Delta\setminus\{\emptyset\}$, for $0 \leq j \leq i$. 

In several parts of this paper we are interested in certain (topological) properties of simplicial complexes such as shellability, 
partitionability and Cohen-Macaulayness. We now recall those notions. 
A pure simplicial complex $\Delta$ is called \textit{shellable} if 
there exists a linear order $F_1,\ldots,F_m$ of the facets of $\Delta$ such that 
$\langle F_i\rangle \cap \langle F_1,\ldots,F_{i-1}\rangle$ is generated by a non-empty set of maximal proper 
faces of $\langle F_i\rangle$ for all $2\leq i\leq m$. Here, $\langle F_i\rangle$ and $\langle F_1,\ldots, F_{i-1}\rangle$, 
denote the simplicial complexes whose faces are all faces of $F_i$ and all faces of any of the $F_1,\ldots,F_{i-1}$, respectively. 
The linear order $F_1,\ldots,F_m$ is called a \textit{shelling} of $\Delta$. 
It is a well-known fact that a shellable simplicial complex $\Delta$ is in particular partitionable. Recall, that $\Delta$ 
is called \emph{partitionable} if $\Delta$ can be written as a disjoint union 
$\Delta=[G_1,F_1]\dotcup\cdots \dotcup [G_m,F_m]$, where $F_1,\ldots,F_m$ are the facets of $\Delta$ and 
$[F,G]=\{H~:~F\subseteq H\subseteq G\}$ is the closed interval from $F$ to $G$. 
Besides being partitionable, shellable simplicial complexes are also known to be Cohen-Macaulay over any field. For our purposes, it will be 
convenient to use the following characterization of the Cohen-Macaulay property due to Reisner. 

\begin{theorem}\cite[Corollary 5.3.9]{BH-book}\label{th:Reisner}
 Let $\Delta$ be a simplicial complex on vertex set $[n]$ and let $k$ be an arbitrary field. Then $\Delta$ is Cohen-Macaulay over $k$ 
if and only if 
\begin{equation*}
\widetilde{H}_i\left(\lk_{\Delta}(F);k\right)=0
\end{equation*}
for all $F\in\Delta$ and all $i<\dim\left(\lk_{\Delta}(F)\right)$.
\end{theorem}

Note that, it follows from this criterion that Cohen-Macaulayness is a topological property. Since $\Delta$ and $\sd(\Delta)$ have homeomorphic
geometric realizations this in particular means that either both complexes are Cohen-Macaulay or none of them is. 
Though the class of shellable simplicial complexes is contained in both, the class of partitionable and the class of Cohen-Macaulay 
complexes, there is no exact relationship known between these two classes. On the one hand, there exists a wide variety of partitionable 
complexes which are not Cohen-Macaulay. On the other hand, it is conjectured that every Cohen-Macaulay complex is partitionable, see e.g., 
\cite{Stanley96}.  
For more information on simplicial complexes we refer the reader to \cite{BH-book,Stanley96}. 

We proceed by recalling some notions from combinatorial topology, see e.g., [11, 17] for a
more thorough treatment of this topic. 
Given a regular cell complex $\Delta$, we call a finite collection $\Delta_1,\ldots, \Delta_l$ of closed subcomplexes 
of $\Delta$ a \emph{covering} $\U$ of $\Delta$ if $\Delta=\Delta_1\cup\cdots\cup\Delta_l$.
The \emph{intersection poset} $P^{\U}$ of the covering $\U$ is the poset whose elements are the intersections $\Delta_J=\bigcap_{i\in J}\Delta_j$, 
where $J\subseteq [l]$,  
which are ordered by reverse inclusion. For $p\in P^{\U}$ we write $U_p$ for the subcomplex of $\Delta$ corresponding to 
the intersection $p$. 
For a poset $P$ and $p\in P$ we let $P_{<p}=\{q\in P~:~q<p\}$ denote the \emph{open lower order ideal} of $p$ in $P$. Similarly, 
$P_{\leq p}$ denotes the \emph{closed lower order ideal} of $p$ in $P$. 
It is common to associate to a poset $P$ its so-called \emph{order complex} $\Delta(P)$, which is the simplicial 
complex on vertex set $P$ whose faces are chains in $P$. 
Note that the barycentric subdivision of a simplicial complex $\Delta$ is the order complex of the face poset of $\Delta$ after the 
removal of the minimum element $\emptyset$.

\subsection{Hypergraphs}
The central object of study of this work are hypergraphs. A simple \emph{hypergraph} $H=(V,E)$ consists of a finite set $V$ of 
\emph{vertices} of $H$ and a collection $E$ of non-empty subsets of $V$, called \emph{edges}. If all edges of $H$ have size two, 
then $H$ is an ordinary graph. We will always assume that 
$H$ has neither isolated vertices nor loops, i.e., edges of cardinality $1$.  
Moreover, we exclude hypergraphs having a pair of edges $F,F'$ such that $F\subsetneq F'$. Mostly, we will consider hypergraphs on vertex set $[n]$. 
If $E'\subseteq E$ is a subset of the edge set of $H=(V,E)$, we define $H_{E'}$ to be the induced subhypergraph of $G$ which has vertex set 
$V_{H_{E'}}=\bigcup_{F\in E'}F$ and edge set $E_{H_{E'}}=E'$. 
A hypergraph $H$ is called \emph{$s$-uniform} if all edges have the same cardinality $s$. 
An \emph{$s$-regular} hypergraph $H=(V,E)$ is a hypergraph such that each vertex $i\in V$ is contained in exactly $s$ edges of $H$. 

We are interested in colorings of hypergraphs and their chromatic polynomials. 
For $k\in\NN$, a \emph{$k$-coloring} of a hypergraph $H=(V,E)$ is just a function $c:V\rar[k]$. Such a $k$-coloring 
$c:V\rar[k]$ is called \emph{proper} if for every edge $F$ there exist vertices $v,w\in F$ such that $c(v)\not=c(w)$. 
All colorings studied in this paper are proper, whence we will often omit this attribute. 
Note that if a hypergraph $H$ has a loop, then $H$ has no proper $k$-colorings for any $k$; therefore we restrict our attention to 
hypergraphs without loops. It is important to emphasize that we require only two vertices of different colors to lie in each edge, 
we do \emph{not} demand all vertices in an edge to have pairwise distinct colors.
\footnote{The latter notion can be captured with proper colorings of ordinary graphs by replacing each edge with a clique on the same vertex set. 
As we wish to study a concept that is strictly more general, we only require edges not to be colored monochromatically.}

Let $H=(V,E)$ be a hypergraph. Consider the function $\chi_H$ that assigns to every $k\in\NN$ the number $\chi_H(k)$ of proper $k$-colorings of $H$. 
Just as in the case of ordinary graphs, $\chi_H(k)$ is a polynomial in $k$, called the \emph{(hypergraph) chromatic polynomial} of $H$ 
(see  e.g., \cite{Birkhoff1912, Wilf}). The fact that $\chi_H$ is a polynomial also follows directly from the geometric considerations in Section~\ref{sec:ehrhart}.

\subsection{Ehrhart theory and geometry}

In this article we consider simplicial complexes not only as abstract combinatorial objects but also as geometric objects. 
Recall that a \emph{polyhedron} in $\RR^n$ is the intersection of finitely many closed halfspaces in $\RR^n$ and that a \emph{polytope} is a bounded polyhedron (see \cite{Ziegler95} for other terminology regarding polyhedra). A \emph{polyhedral complex} is a set 
$\C$ of finitely many polyhedra in some $\RR^n$ such that if $P,Q\in\C$, then $P\cap Q\in \C$ and $P\cap Q$ is a face of both $P$ and $Q$. 
A \emph{polytopal complex} is a polyhedral complex in which all faces are polytopes and a \emph{(geometric) simplicial complex} is a polytopal 
complex in which all faces are simplices. Every geometric simplicial complex induces an abstract simplicial complex on its vertex set. 
The \emph{support} of a polyhedral complex $\C$ is $\bigcup_{P\in\C} P$, i.e., the underlying subset of $\RR^n$. A \emph{subdivision} 
$\C'$ of a polyhedral complex $\C$ is a polyhedral complex with $\bigcup_{P'\in\C'} P'= \bigcup_{P\in\C} P$ such that every 
$P'\in\C'$ is contained in some $P\in\C$. If $\C'$ is simplicial, then $\C'$ is also called a \emph{triangulation}. 
The \emph{intersection} of two polyhedral complexes $\C,\C'$ is the polyhedral complex $\C\cap\C'=\{ P\cap P': P\in\C, P'\in\C'\}$. 

A \emph{subspace arrangement} $\A$ is a finite collection of affine subspaces in some $\RR^n$. A \emph{hyperplane arrangement} 
$\H$ is a subspace arrangement in which every subspace is an affine hyperplane. A frequently and well-studied hyperplane 
arrangement is the so-called \emph{braid arrangement} in $\RR^n$, which is the collection of hyperplanes 
$\{H_{ij}~:~ i,j\in[n], i\not=j\}$, where $H_{ij}=\{x\in\RR^n~:~ x_i=x_j\}$. 
Every hyperplane arrangement $\H$ induces a polyhedral complex $\C_\H$ which is a subdivision of $\RR^n$. 
Given a polyhedral complex $\C$ and a hyperplane arrangement $\H$, the \emph{subdivision of $\C$ induced by $\H$} is 
the intersection of $\C$ and $\C_\H$. 

For every $X\subset\RR^n$, the \emph{Ehrhart function} $L_X:\NN\rar\NN$ assigns to every $k\in \NN$ the number $L_X(k)=\#(k\cdot X \cap \ZZ^n)$ 
of integer points in the $k$-th dilate of $X$. A \emph{lattice polytope} is a polytope whose vertices have only integer coordinates. 
It is a fundamental result of Ehrhart that if $X$ is a lattice polytope, then $L_X(k)$ is a polynomial in $k$, or, more precisely, 
there is a polynomial $p(k)$ such that $L_X(k)=p(k)$ for every $k\in\NN$, see \cite[Theorem 3.8]{BeckRobins07}. 

Two polytopes $P,Q$ are \emph{lattice equivalent} if there exists a an isomorphism $f\in GL(n,\ZZ)$ with $P=f(Q)$. 
Lattice equivalent polytopes have the same Ehrhart function. A $d$-simplex is \emph{unimodular} if it is lattice equivalent to a standard simplex. 
Here, a \emph{standard simplex} in $\RR^n$ refers to a simplex whose vertex set is a subset of the $n$ standard unit vectors in $\RR^n$ and the origin. 
Every abstract simplicial complex can be realized as a geometric simplicial complex in which every simplex is unimodular. Such a geomoetric realization 
will be referred to as \emph{unimodular}. 
When we speak of the Ehrhart function of an abstract simplicial complex $\Delta$, we mean the Ehrhart function of any unimodular geometric realization of $\Delta$; the Ehrhart function of a unimodular realization is independent of the particular choice of unimodular realization.

\section{The combinatorial hypergraph coloring complex}\label{subsect:combCol}
\label{sec:combinatorial}

In this section, we introduce the so-called \emph{(combinatorial) coloring complex} associated to a hypergraph and investigate some of its 
properties. 
The given construction is a natural generalization of the coloring complex of a graph, see e.g., 
\cite{Hultman,Jonsson,Steingrimsson}. In particular, for an ordinary graph we rediscover its usual coloring complex. 
Though the latter one has been shown to exhibit fairly nice properties, e.g., shellability, this is no longer true in general when passing to 
arbitrary hypergraphs. 

Let $\Pa\left([n]\right)$ denote the set of ordered set partitions of $[n]$ having no empty block. 
We define an ordering relation $\preceq$ on $\Pa\left([n]\right)$ in the following way.
A partition $B=B_1|\cdots|B_r$ covers exactly those partitions which can be obtained by 
taking the union of two adjacent blocks of $B$, i.e., all partitions $B_1|\cdots|B_{i-1}|B_i\cup B_{i+1}|B_{i+2}|\cdots|B_r$ for $1\leq i\leq r-1$.
It is straightforward to verify that -- endowed with this ordering relation -- each interval in $\Pa\left([n]\right)$ is isomorphic to a Boolean lattice. 
Moreover, $\Pa\left([n]\right)$ has a minimum element which is the partition consisting of the single block $[n]$.

We now state the definition of the combinatorial hypergraph coloring complex. Equivalent definitions in geometric terms are given in Sections~\ref{subsect:arrangement} and \ref{subsect:ColEhrhart}.

\begin{definition}
Let $H=([n],E)$ be a hypergraph. The \emph{(combinatorial) hypergraph coloring complex} $\Delta_H$ associated to $H$ is the 
simplicial complex whose $(r-2)$-dimensional faces are set partitions $B_1|B_2|\cdots|B_r$ of $\Pa([n])$ such that 
there exists at least one block $B_i$ (for some $1\leq i\leq r$) containing an edge of $H$. 
The containment relation between two faces is defined via the ordering $\preceq$ on $\Pa([n])$. 
\end{definition}

It directly follows from the definition that facets of $\Delta_H$ are those set partitions which are comprised of 
one block equal to a certain edge and singleton blocks otherwise. 
We make the following two fundamental observations for the coloring complex of a hypergraph $H=([n],E)$. 

\begin{remark}\label{rem:pure}
 \begin{itemize}
  \item[(i)] Let $m=\min\{\#F~:~F\in E\}$ be the minimal cardinality of an 
  edge of $H$. Then, the dimension of $\Delta_H$ equals $n-m-1$.
\item[(ii)] Let $H=([n],E)$ be a hypergraph having no pair of edges such that one is properly contained in the other. Then, 
$\Delta_H$ is a pure complex of dimension $n-1-s$ if and only if $H$ is $s$-uniform.
 \end{itemize}
\end{remark}

We will now consider a few simple examples of hypergraph coloring complexes. Those will also be used to fix 
some further notation.

\begin{example} \label{Ex1}
\begin{itemize}
  \item[(i)] If $H=([n],E)$ is an ordinary graph, then it directly follows from the definition that the 
hypergraph coloring complex $\Delta_H$ coincides with the usual coloring complex, which was introduced by Steingr\'{i}msson in \cite{Steingrimsson}.
  \item[(ii)] Consider a hypergraph $H=([n],E)$ which consists of just one edge and isolated vertices otherwise. For instance, let
 $E=\{[s]\}$. Then, the partition 
\begin{equation*}
 A=[s]|\{s+1\}|\cdots|\{n-1\}|\{n\}
\end{equation*}
defines a facet of $\Delta_H$ and any facet of $\Delta_H$ can be obtained from $A$ by permuting the order of its blocks. Moreover, each 
such reordering of the blocks of $A$ yields a facet of $\Delta_H$. Thus, $\Delta_H$ has exactly $(n-s+1)!$ facets. The same arguments as in 
the proof of Theorem 14 in \cite{Steingrimsson} show that -- as a simplicial complex -- $\Delta_H$ is isomorphic to the barycentric subdivision of 
the boundary of an $(n-s)$-simplex and as such is homeomorphic to an $(n-s-1)$-sphere.
\end{itemize}
\end{example}

For a hypergraph $H=(V,E)$ and any edge $F\in E$, we set
\begin{equation}
Q_F:=\{B_1|\cdots|B_r\in \Delta_H~:~\mbox{there exists }1\leq i\leq r \mbox{ such that } F\subseteq B_i\},
\label{eq:edgesphere}
\end{equation}
i.e., $Q_F$ is the set of those faces of the coloring complex $\Delta_H$ which have a block containing $F$.
By definition, $Q_F$ is an $(n-|F|-1)$-dimensional subcomplex of $\Delta_H$, and in Example \ref{Ex1} (ii) we have seen 
that $Q_F$ is homeomorphic to a sphere. Following the notions in \cite{Steingrimsson}, we will refer to such a sphere as an 
\emph{edge sphere}.

\subsection{An arrangement interpretation} \label{subsect:arrangement}

It was shown in \cite[Theorem 1]{HershSwartzCol} that the coloring complex of an ordinary graph can be interpreted in terms 
of certain hyperplane arrangements. The aim of this section is to carry this description over to the coloring 
complex of a hypergraph. The main difference -- though not an astonishing one -- is that arrangements consisting of affine linear spaces of arbitrary dimension, and not just hyperplane arrangements, come into play. We will strongly make use of a result from \cite{HRW-Koszul}.
Before stating this result we need to fix some notation and establish some basics. 

Given a square-free monomial $m=x_{i_1}\cdots x_{i_s}\in \RR[x_1,\ldots,x_n]$, we can assign a linear subspace 
$U_m$ of $\RR^n$ to it by setting $U_m=\{(u_1,\ldots,u_n)\in\RR^n~:~u_{i_1}=\cdots=u_{i_s}\}$. 
In the following, we will give two constructions which associate to a monomial ideal a certain subspace 
arrangement and a simplicial complex, respectively. 
Let $J\subsetneq \RR[x_1,\ldots,x_n]$ be a monomial ideal and consider those minimal generators of $J$ which are square-free, 
say $m_1,\ldots,m_t$. 
The \emph{canonical arrangement} $\A_J$ corresponding to $J$ is the subspace arrangement inside the hyperplane 
$H=\{u=(u_1,\ldots,u_n)\in\RR^n~:~u_1+\cdots +u_n=0\}$ consisting of its intersection with the union of all linear subspaces $U_{m_i}$, 
for $1\leq i\leq t$.

Moreover, as described in \cite{HRW-Koszul}, one can associate to the monomial ideal $J\subsetneq \RR[x_1,\ldots,x_n]$ 
a simplicial complex $\Delta_J$ on vertex set $2^{[n]}-\{\emptyset\}$ in the following way. The 
$(l-1)$-faces of $\Delta_J$ are chains
\begin{equation}\label{eq1}
 \emptyset\neq A_1\subsetneq A_2\subsetneq \cdots \subsetneq A_l\neq [n],
\end{equation}
such that $x^{A_i\setminus A_{i-1}}=\prod_{r\in (A_i\setminus A_{i-1})}x_r\in J$ for some $1\leq i\leq l+1$. 
Here, we set $A_0=\emptyset$ and $A_{l+1}=[n]$. 
By definition, the complex $\Delta_J$ is a subcomplex of the barycentric subdivision of the boundary of an $(n-1)$-simplex. 
In the following, we denote this subdivision by $\sd(\partial\Delta^{n-1})$. 
The next theorem is a special case of Theorem 3.1 in \cite{HRW-Koszul}.

\begin{theorem}\cite[Theorem 3.1]{HRW-Koszul}\label{th:HRW}
Let $J\subsetneq \RR[x_1,\ldots,x_n]$ be a monomial ideal. Then the pair $\left(\sd(\partial\Delta^{n-1}),\Delta_{J}\right)$ 
is homeomorphic to the pair $\left(S^{n-2},S^{n-2}\cap \A_J\right)$, where $S^{n-2}$ is the unit sphere in the hyperplane 
$H=\{u=(u_1,\ldots,u_n)\in\RR^n~:~u_1+\cdots +u_n=0\}$ and $\mathcal{A}_J$ is the canonical arrangement corresponding to $J$.
\end{theorem}

We will now explain how this result serves our purposes.

Let $H=([n],E)$ be a hypergraph and -- as usual -- assume that none of its edges is properly contained in any other edge. 
The \emph{edge ideal} of $H$ is the monomial ideal $I_H\subsetneq \RR[x_1,\ldots,x_n]$ generated by the monomials 
$x_F=\prod_{i\in F}x_i$, where $F\in E$ is an edge. 
Since a chain as in (\ref{eq1}) can be converted into an ordered set partition 
$A_1|(A_2\setminus A_1)|\cdots |(A_l\setminus A_{l-1})|[n]\setminus A_l$ of $[n]$ 
and vice versa, it follows directly from the definition that $\Delta_{I_H}$ is simplicially isomorphic to the 
hypergraph coloring complex $\Delta_{H}$.
Accessorily, the coloring complex $\Delta_{K_n}$ of the complete graph $K_n$ on $n$ vertices is known to be simplicially isomorphic 
to $\sd(\partial\Delta^{n-2})$, and the canonical arrangement corresponding to $I_{K_n}$ is the usual 
\emph{braid arrangement} in $\RR^n$. 

Combining this argumentation with Theorem \ref{th:HRW} and using the same arguments as in \cite[Theorem 1]{HershSwartzCol} we obtain the desired interpretation of hypergraph coloring complexes in terms of arrangements.

\begin{theorem}\label{SubspaceArrangementInterpretation}
Let $H=([n],E)$ be a hypergraph. 
As a simplicial complex, the hypergraph coloring complex $\Delta_H$ of $H$ is isomorphic 
to the restriction of $\mathcal{A}_{n}\cap S^{n-2}$ to $\A_H=\A_{I_H}$, where $S^{n-2}$ is the unit sphere in the hyperplane 
$\{u=(u_1,\ldots,u_n)\in\RR^n~:~u_1+\cdots +u_n=0\}$ and $\mathcal{A}_{I_H}$ is the canonical arrangement corresponding to $I_H$. Moreover, the pair $(\Delta_{K_n},\Delta_{H})$ is homeomorphic to the pair $(S^{n-2},S^{n-2}\cap\mathcal{A}_H)$.
\end{theorem}

\subsection{The Cohen-Macaulay property}

It was shown by Jonsson \cite[Theorem 1.4]{Jonsson} that coloring complexes of graphs are constructible and in particular (homotopy) Cohen-Macaulay. 
Hultman \cite[Theorem 4.2]{Hultman} strengthened this result by providing a proof that those complexes are shellable. 
More precisely, he constructed a shelling for so-called \emph{link complexes} $\Delta_{\mathcal{A},\mathcal{H}}$, where $\mathcal{H}$ is a simplicial hyperplane 
arrangement and $\mathcal{A}$ a subspace arrangement consisting of hyperplanes. If $\mathcal{H}$ is the braid arrangement and $\mathcal{A}$ the 
subarrangement given by the edges of a graph $G$ (see Section \ref{subsect:arrangement}), the link complex $\Delta_{\mathcal{A},\mathcal{H}}$ coincides with the coloring complex of $G$. 
If $G$ is a connected graph, then shellability of $\Delta_G$ also follows from \cite[Remark 6]{HershSwartzCol}.

One could hope that maybe under some additional assumptions the same result holds in the more general situation of hypergraphs. However, as we 
will show, hypergraph coloring complexes tend to behave rather badly. More precisely, they are not even Cohen-Macaulay in general.

\begin{proposition}\label{prop:nonCM}
Let $H=([n],E)$ be a hypergraph having at least one edge 
of cardinality greater than $2$. Assume that $H$ has two disjoint edges. Then $\Delta_H$ is not Cohen-Macaulay over any field.
\end{proposition}

\begin{proof}
If $H$ is not a uniform hypergraph, then it follows from Remark \ref{rem:pure} 
that $\Delta_H$ is not pure and hence, not Cohen-Macaulay. 

So, assume that $H$ is uniform and 
let $s\geq 3$ be the cardinality of any edge of $H$. By assumption, there exist edges $F_1$, $F_2\in E$ such that 
$F_1\cap F_2=\emptyset$. Without loss of generality, we may further assume that $F_1=\{1,\ldots,s\}$ and $F_2=\{s+1,\ldots,2s\}$.

Consider the face $B=[s]|\{s+1,\ldots,2s\}|\{2s+1\}|\cdots|\{n\}$ of $\Delta_H$. 
In the following, we will compute the link of $B$ in $\Delta_H$ and show that it is disconnected. 
For this aim, we first determine the facets $A_1|\cdots|A_{n-s+1}$ of $\Delta_H$ which contain $B$ as a face. 
We distinguish the following two types of those facets:
\begin{itemize}
 \item[Type I:] $A_1=[s]$, $A_i=\{\sigma(s+i-1)\}$ for $2\leq i\leq s+1$ and a permutation $\sigma$ of $\{s+1,\ldots,2s\}$ and $A_i=\{i+s-1\}$ 
for $s+2\leq i\leq n-s+1$
\item[Type II:]  $A_i=\{\sigma(i)\}$ for $1\leq i\leq s$ and a permutation $\sigma$ of $[s]$, $A_2=\{s+1,\ldots,2s\}$ and $A_i=\{i+s-1\}$ 
for $s+2\leq i\leq n-s+1$.
\end{itemize}
As defined in (\ref{eq:edgesphere}), let $Q_{F_1}$ and $Q_{F_2}$ denote the set of faces of $\Delta_H$ having a block containing $F_1$ and 
$F_2$, respectively. Then all facets of type I and type II are contained in $Q_{F_1}$ and $Q_{F_2}$, respectively. 
In particular, each face of $\lk_{\Delta_H}(B)$ lies in $Q_{F_1}$ or $Q_{F_2}$, i.e.,
\begin{equation*}
 \lk_{\Delta_H}(B)=\lk_{Q_{F_1}}(B)\cup\lk_{Q_{F_2}}(B).
\end{equation*}
Moreover, if $S$ and $T$ are facets of $Q_{F_1}$ and $Q_{F_2}$, respectively, and if $B$ is a face of both, $S$ and $T$, then $S\cap T= B$. From this we infer that 
$\lk_{Q_{F_1}}(B)\cap \lk_{Q_{F_2}}(B)=\emptyset$. Since neither of those links is empty, we conclude 
that the link of $B$ in $\Delta_H$ is disconnected.
It follows from our assumptions that $\dim(\lk_{\Delta_H}(B))=s-2\geq 1$ and using Reisner's criterion (Theorem \ref{th:Reisner}) we conclude that 
$\Delta_H$ is not Cohen-Macaulay.
\end{proof}

We now consider an example which illustrates the idea of the above proof.

\begin{example}
Let $H=([6],E)$ where $E=\{\{1,2,3\},\{2,3,4\},\{2,4,5\},\{4,5,6\}\}$. In particular, the edges $\{1,2,3\}$ and 
$\{4,5,6\}$ are disjoint. Consider the vertex $B=\{1,2,3\}|\{4,5,6\}$ of $\Delta_H$. All facets of $\Delta_H$ containing $B$ 
either have $\{1,2,3\}$ as their first block followed by the singletons $\{4\}$, $\{5\}$ and $\{6\}$ in some order or they 
have $\{4,5,6\}$ as their last block preceded by the singletons $\{1\}$, $\{2\}$ and $\{3\}$ in some order. Note that 
in the proof of Proposition \ref{prop:nonCM} the former and the last ones are called type I and type II facets, respectively. 
The link of $B$ in $\Delta_H$ is given as 
\begin{equation*}
\lk_{\Delta_H}(\{1,2,3\}|\{4,5,6\})=\lk_{Q_{\{1,2,3\}}}(\{1,2,3\}|\{4,5,6\})\cup \lk_{Q_{\{4,5,6\}}}(\{1,2,3\}|\{4,5,6\}),
\end{equation*}
and it is easy to verify that it consists of two disjoint $6$-cycles. Thus, by Theorem \ref{th:Reisner}, $\Delta_H$ is not 
Cohen-Macaulay.
\end{example}

\section{Ehrhart theory, the chromatic polynomial and enumerative consequences}
\label{sec:ehrhart}

In this section we will examine coloring complexes from the point of view of Ehrhart theory and employ this approach to draw some enumerative conclusions regarding the $f$- and $h$-vectors of the coloring complex as well as the coefficients of the chromatic polynomial of a hypergraph.

\subsection{The coloring complex from the point of view of Ehrhart theory}\label{subsect:ColEhrhart}

The coloring complex of an ordinary graph can be studied from the perspective of inside-out polytopes \cite{breuer-diss,breuer-dall2,breuer-dall1,breuer-sanyal}. In this section we extend this approach to hypergraph coloring complexes.

The braid arrangement triangulates the unit cube $[0,1]^n\subset\RR^n$ into a simplicial complex $C^n$. 
Let $V([0,1]^n)$ denote the vertex set of $[0,1]^n$. Note that any vertex $A$ of $[0,1]^n$ can be interpreted as a subset of $[n]$, 
whence inclusion induces a partial order $\subset$ on $V([0,1]^n)$. The set $V([0,1]^n)$ ordered by $\subset$ forms precisely the 
Boolean lattice on $n$ atoms, with minimal element the all-zero vector $\mathbf{0}$ and maximal element the all-one vector $\mathbf{1}$. 
When $C^n$ is viewed as an abstract simplical complex on ground set $V([0,1]^n)$, the faces of $C^n$ are in bijection with the chains 
in $(V([0,1]^n),\subset)$: More precisely, $\{A_1,\ldots,A_l\} \subset V([0,1]^n)$ is an $(l-1)$-face of $C^n$ if and only if 
$A_1 \subsetneq \cdots \subsetneq A_l$ forms a chain in $(V([0,1]^n),\subset)$.

As in the arrangement interpretation of the coloring complex, every edge $F$ of a hypergraph $H=([n],E)$ corresponds to a linear subspace 
$H_F=\{x\in\RR^n ~:~ x_v=x_w ~\forall v,w\in F\}$. For all $F\subseteq [n]$ we let
\[
  P_F := H_F\cap [0,1]^n \text{ and } \Box_H := \bigcup_{F\in E} P_F.
\]
By abuse of notation we will denote by $P_F$ and $\Box_H$ both the subsets of $[0,1]^n$ defined above as well as the subcomplexes of $C^n$ they induce. (For example, the subcomplex of $C^n$ induced by $\Box_H$ consists of all faces of $C^n$ that are contained, as a subset of $\mathbb{R}^n$, in $\Box_H$.)

The theory of inside-out polytopes \cite{BeckZaslavsky06a,BeckZaslavsky06b} gives rise to the immediate observation that the Ehrhart function 
of $[0,1]^n\setminus \Box_H$ equals the chromatic polynomial $\chi_H$ of $H$ shifted by one: 
As a consequence of a theorem of Ehrhart, c.f.~\cite[Theorem 3.8]{BeckRobins07}, $L_X(k)$ is a polynomial in $k$ if $X$ 
is a unimodular simplicial complex. In the case of hypergraph colorings, we observe that the integer points 
$x\in \left(\ZZ^n \cap (k\cdot [0,1]^n \setminus \Box_H)\right)$ can be interpreted as colorings of the vertices of 
$H$ with $k+1$ colors such that for every edge $F$ there exist vertices $v,w\in F$ with $x_v\not=x_w$, i.e., the colorings are proper. We conclude that 
\[ 
  L_{[0,1]^n\setminus \Box_H}(k) = \chi_H(k+1).
\]

Now we relate this construction to the coloring complex. Note that every $P_F$ contains both $\mathbf{0}$ and $\mathbf{1}$. 
For a hypergraph $H=([n],E)$ and a set $F\subset[n]$ we define the complexes $Q_F$ and $\Delta_H$ as follows.
\[
  Q_F := P_F\setminus\{\mathbf{0},\mathbf{1}\} \text{ and } \Delta_H := \bigcup_{F\in E} Q_F
\]
Here $P_F\setminus\{\mathbf{0},\mathbf{1}\}$ denotes the complex consisting of all faces of $P_F$ that do not contain the vertex $\mathbf{0}$ and that do not contain the vertex $\mathbf{1}$. As it turns out, the complexes $Q_F$ are precisely the edge spheres defined in Section~\ref{sec:combinatorial}: A chain 
\[
  \emptyset \not= A_1 \subsetneq \cdots \subsetneq A_l \not= [n]
\]
of length $l-1$ is an $(l-1)$-face of $Q_F$ if and only if for every $1\leq i\leq l$ the set $F$ is either disjoint from $A_i$ or contained in $A_i$. 
This condition is equivalent to the property that there exists $1\leq i \leq l+1$ such that $F\subseteq A_{i}\setminus A_{i-1}$, where 
$A_0=\emptyset$ and $A_{l+1}=[n]$ as above. Consequently, the abstract simplicial complexes $Q_F$ are isomorphic to the edge spheres and 
the abstract simplicial complex $\Delta_H$ is isomorphic to the hypergraph coloring complex as defined previously. (The notation $Q_F$ and $\Delta_H$ 
is thus unambiguous.)

\subsection{$f$- and $h$-vectors of polynomials and complexes}

The $f$- and $h$-vectors of a polynomial $p(k)$ are coefficient vectors with respect to certain bases of the vector space of polynomials. 
Consider a positive integer $n$ and a polynomial $p(k)$ of degree at most $n$. The \emph{$f$-vector} $f(p)=(f_{-1},f_0,\ldots,f_n)$ of $p(k)$ is defined by 
\[
    p(k) =\sum_{i=0}^{n} f_i {k-1 \choose i}
\]
and $f_{-1}=1$. 
Here we use the fact that the polynomials ${k-1\choose i}$, $0\leq i\leq n$ form a basis of the vector space of polynomials of degree at most $n$. 
Similarly, we define the \emph{$h$-vector} $h(p)=(h_0,\ldots, h_{n+1})$ of $p(k)$ by
\[
    p(k) ={k+n \choose n} + \sum_{i=1}^{n+1} h_i {k+n-i \choose n}.
\]
and $h_0=1$. Here we use the fact that the polynomials ${k+n-i \choose n}$ for $1\leq i\leq n+1$ form a basis of the vector space of polynomials of 
degree at most $n$. The $f$- and $h$-vectors are related by
\begin{eqnarray}
\label{eqn:f-vs-h}
  h_i =  \sum_{k=-1}^{i-1} (-1)^{i-k-1} {n-k \choose i-k-1} f_{k}.
\end{eqnarray}
for $0\leq i\leq n+1$.

Note, that as long as $n\geq\deg(p(k))$, the value of $f_i$ is independent of the choice of $n$. If $n$ is chosen to be larger, 
zeros are appended to the end of the $f$-vector of $p$. This is not true for the $h$-vector. 
If the length of the $h$-vector is chosen differently, all entries of the $h$-vector will change, in general. 
If we wish to emphasize the parameter $n$ with respect to which the $h$-vector is defined, we denote the entries of the $h$-vector by $h^n_i$ for $0\leq i\leq n+1$. 

$f$- and $h$-vectors are classical parameters of simplicial complexes \cite{Stanley96,Ziegler95}. As stated in Section \ref{subsect:simplicial} the
$h$-vector of a simplicial complex $\Delta$ can be obtained as a transformation of the $f$-vector. It is a direct consequence of Equation 
(\ref{eq:ftoh}) in Section \ref{subsect:simplicial} that 
$h^{\Delta}$ can be computed via the formula given in (\ref{eqn:f-vs-h}).

The link between these two notions of $f$- and $h$-vectors is given by Ehrhart theory. If $\Delta$ is an $n$-dimensional geometric simplicial 
complex in which all simplices are unimodular and if $L_\Delta$ is its Ehrhart polynomial, 
then $f(L_\Delta)$ and $h^n(L_\Delta)$ are the $f$- and $h$-vectors, 
respectively, of the abstract simplicial complex $\Delta$. See \cite{breuer-dall2} for details.

From this point of view it is also straightforward to prove the relationship (\ref{eqn:chi-and-h}) between the chromatic polynomial of a hypergraph and the $h$-vector of the coloring complex as given in the introduction: If $n$ is the number of vertices of $H$ and the maximal cardinality of an edge of $H$ is $m$, then $\Delta_{H}$ is a simplical complex of dimension $n-m-1$ with $h$-vector $(h_0,\ldots,h_{n-m})$. $\Box_H$ is of dimension $n-m+1$ but has the same $h$-vector, as $\Box_H$ is the double cone over $\Delta_H$. We have already seen $L_{[0,1]^n\setminus\Box_H}(k)=\chi_H(k+1)$, which is equivalent to $(k+1)^n-\chi_H(k+1)=L_{\Box_H}(k)$. Passing to series, we compute
\begin{eqnarray*}
\frac{1}{z}\sum_{k\geq 1}\left((k)^n - \chi_H(k)\right) z^k
&=&
\sum_{k\geq 0}\left((k+1)^n - \chi_H(k+1)\right) z^k
\\ &=&
\sum_{k\geq 0}L_{\Box_H}(k) z^k
\\ &=&
\frac{h_0z^0+\ldots+h_{n-m}z^{n-m}}{(1-z)^{n-m+2}}
\end{eqnarray*}
where we use that $\sum_{k\geq 0}L_{\Box_H}(k) z^k$ is the Ehrhart series of the complex $\Box_H$, whence the coefficients of the numerator of $\frac{h_0z^0+\ldots+h_{n-m}z^{n-m}}{(1-z)^{n-m+2}}$ form the $h$-vector of $\Box_H$, see \cite[Chapter 3]{BeckRobins07}. Finally, we note that for any hypergraph $H$ we can start the series on the left-hand side at $k=0$ because $\chi_H(0)=0$. This is easiest to see via a slightly different construction: $\chi_H(k)=L_{(0,1)^n\setminus\Box_H}(k+1)$ and $\chi_H(0)=L_{(0,1)^n\setminus\Box_H}(1)=0$ as $(0,1)^n$ does not contain any lattice points.

\subsection{The combinatorics of the complexes $P_F$}

For the enumerative computations that follow, it is crucial to observe that the complexes $P_F$ are unit cubes triangulated by the braid arrangement. 
We also compute their $f$- and $h$-vectors. To simplify notation, we will use $f(P_F)$ and $h^{n}(P_F)$ to denote the vectors $f(L_{P_F})$ and 
$h^n(L_{P_F})$, respectively, and similarly for other complexes. 

\begin{proposition}
\label{prop:f-of-cube}
If $F\subseteq [n]$ and $\#F= k$, then $P_F$ is a unimodular simplicial complex isomorphic to 
the braid triangulation $C^{n-k+1}$ of an $(n-k+1)$-cube. Moreover, for all $0\leq i \leq n-k+1$
\begin{eqnarray}
\label{eqn:f-of-cube}
  f_i(P_F) & = & \sum_{j=0}^i (-1)^j {i \choose j} (i-j+2)^{n-k+1}.
\end{eqnarray}
Finally, if $n'\geq n-k+1$ and $0\leq i\leq n'+1$ then
\begin{eqnarray}
\label{eqn:h-of-cube}
  h^{n'}_i(P_F) & = & (-1)^{i} {n'+1 \choose i} + \sum_{a=0}^{i-1}\sum_{b=0}^a (-1)^{i-a+b-1} {n'-a \choose i-a-1} {a \choose b} (a-b+2)^{n-k+1}
\end{eqnarray}
\end{proposition}

\begin{proof}
The idea for the construction of the isomorphism is simply to contract the edge $F$. Without loss of generality, we can assume that $F=\{n-k+1,\ldots,n\}$. Let $V(P_F)$ and $V(C^{n-k+1})$ denote the vertex sets of 
$P_F$ and $C^{n-k+1}$, respectively. We define a map $\phi$ from $V(P_F)$ to $V(C^{n-k+1})$ as follows. 
For any vertex $A\subseteq [n]$ of $P_F$ we let $\phi(A):=A\setminus(F\setminus\{n-k+1\})\subseteq [n-k+1]$. 
Note that $n-k+1\in\phi(A)$ if and only if $F\subset A$ and $n-k+1\not\in\phi(A)$ if and only if $F\cap A=\emptyset$. 
It is straightforward to verify that $\phi$ gives an isomorphism between $P_F$ and $C^{n-k+1}$.

To see (\ref{eqn:f-of-cube}), observe, on the one hand, that the number $T(d,k)$ of chains of length $k$ in a Boolean lattice on $d$ atoms is
\[
  T(d,k) = \sum_{j=0}^k (-1)^j {k \choose j}(k-j+2)^d,
\]
by \cite[A038719]{oeis}. On the other hand, it follows from Section \ref{subsect:ColEhrhart} that $f_i(P_F)=T(n-k+1,i)$, which yields the desired identity.

To compute the $h$-vector of $P_F$, we apply the transformation (\ref{eqn:f-vs-h}) to (\ref{eqn:f-of-cube}). 
Let $d$ be the dimension of the cube $C^d$ and let $n'\geq d$ be the parameter of the $h$-vector. Then, using that $f_{-1}=1$, we obtain 
\begin{eqnarray*}
  h^{n'}_i(C^d) & = & (-1)^{i} {n'+1 \choose i} + \sum_{a=0}^{i-1} (-1)^{i-a-1} {n'-a \choose i-a-1} f_{a}(C^d) \\
&=& (-1)^{i} {n'+1 \choose i} + \sum_{a=0}^{i-1}\sum_{b=0}^a (-1)^{i-a+b-1} {n'-a \choose i-a-1} {a \choose b} (a-b+2)^{d}.
\end{eqnarray*}
Applying the fact that $P_F$ is isomorphic to $C^{n-k+1}$ completes the proof.
\end{proof}

Using an analogous proof, one can also show the following statement.

\begin{remark}
\label{rem:intersection-of-P_F}
Let $S\subseteq E$ be a subset of the edge set of the hypergraph $H=([n],E)$. Let $b$ be the number of connected components of the restricted hypergraph $([n],S)$. Then $\bigcap_{F\in S} P_F$ is isomorphic to $C^b$. Moreover, for all $0\leq i \leq b$
\begin{eqnarray}
  f_i(\bigcap_{F\in S} P_F) & = & \sum_{j=0}^i (-1)^j {i \choose j} (i-j+2)^{b}.
\end{eqnarray}
\end{remark}

The idea of the proof is, again, to contract all the components of $([n],S)$.

\subsection{The $f$-vector of the chromatic polynomial}

In \cite{HershSwartzCol} Hersh and Swartz give bounds on the coefficients of the chromatic polynomial of a graph by giving bounds on the $h$-vector of a suitable transformation of the chromatic polynomial. The crucial ingredient of the proof is that coloring complexes of graphs have \emph{convex ear decompositions}. This is not true for hypergraphs, as we see for example from the fact that coloring complexes of hypergraphs can have negative entries in their $h$-vector, see Example~\ref{ex:not-pertitionable}. Nonetheless, it is possible to obtain bounds on the coefficients of the chromatic polynomial of a hypergraph. These are most conveniently expressed in terms of the $f$-vector of the chromatic polynomial.

The fact that $f_i(\chi_H(k+1))$ counts the number of $i$-dimensional faces in $C^n$ that are not contained in $\Box_H$ yields a number of useful results. In particular, it allows the elementary observation that $\chi_H(k)\leq \chi_{H'}(k)$ for a hypergraph $H$ and a subgraph $H'$ to be strengthened in two ways.

\begin{theorem}
\label{thm:subgraph}
Let $H=([n],E)$ be a hypergraph and $H'$ a subgraph. Let $H^*=([n^*], E^*)$ be any hypergraph with the property that for every edge $F\in E$ 
there exists an edge $F^*\in E^*$ such that $F^* \subseteq F$. Then for all $0\leq i \leq n$
\[
 0 \leq f_i(\chi_{H^*}(k)) \leq f_i(\chi_H(k)) \leq f_i(\chi_{H'}(k)).
\]
\end{theorem}

Note that, for any polynomials $p(k)$ and $q(k)$ of degree at most $n$ we always have that if $f_i(p(k)) \leq f_i(q(k))$ for all $i$, then $p(k)\leq q(k)$ for all $k>0$ as the binomial coefficients ${k-1 \choose i}$ take non-negative values for positive $k$. 

\begin{proof}
Let $\Box'$ denote the subcomplex of $\partial C^n$ induced by the set $[0,1]^n\setminus [0,1)^n$. Observe that $\chi_H(k)=L_{[0,1]^n\setminus (\Box_H \cup \Box')}(k)$, where $\Box'$ does not depend on $H$. Therefore, 
we can prove an inequality of the form $f_i(\chi_{H_1}(k))\leq f_i(\chi_{H_2}(k))$ by proving that $\Box_{H_1} \supseteq \Box_{H_2}$.

The fact that $f_i(\chi_H(k))$ equals the number of $i$-dimensional faces in $C^n$ that are not contained in $\Box_H\cup\Box'$ shows the first inequality.

If $\sigma$ is a face of $\Box_{H}$, then $\sigma$ is contained in a linear subspace $H_F$ for some edge $F$ of $H$. 
Then, there exists an edge $F^*$ of $H^*$ such that $F^* \subseteq F$ and $H_F \subseteq H_{F^*}$. Thus $\sigma$ is also a face of $\Box_{H^*}$. This shows the second inequality.

If $\sigma$ is a face of $\Box_{H'}$, then $\sigma$ is contained in a linear subspace $H_F$ for some edge $F$ of $H'$. As $F$ is also an edge of $H$, it follows that $\sigma$ is a face of $\Box_H$. This shows the last inequality.
\end{proof}

Let $H=([n],E)$ be a hypergraph and let $S\subseteq E$ be a subset of the edge set. 
For all $1\leq a \leq \#E$ and $b\in \NN$ we denote by $s(a,b)$ the number of sets $S\in {E \choose a}$ such that $([n],S)$ has $b$ components. 
For convenience, we define $s(0,n)=1$ and $s(0,b)=0$ for all $b\not=n$, independent of the edge set $E$. 
Note that since we only consider hypergraphs without loops, it holds that $s(a,n)=0$ for all $a\geq 1$.

\begin{theorem}
\label{thm:f-vector}
Let $n$ be a positive integer and $H=([n],E)$ be a hypergraph without loops. Then for all $0\leq i\leq n$ and every $0\leq m \leq \#E$
\begin{eqnarray}
f_i(\chi_H(k+1)) & = & 
\sum_{a=0}^{\#E} (-1)^a\sum_{b=0}^n s(a,b)\cdot \left( \sum_{c=0}^i(-1)^c {i \choose c}(i-c+2)^{b} \right), \label{eqn:f-v-eq}\\
f_i(\chi_H(k+1)) & \leq & 
\sum_{a=0}^{m} (-1)^a\sum_{b=0}^n s(a,b)\cdot \left( \sum_{c=0}^i(-1)^c {i \choose c}(i-c+2)^{b} \right), \;\;\;\;\text{if $m$ is even}, \label{eqn:f-v-even}\\
f_i(\chi_H(k+1)) & \geq & 
\sum_{a=0}^{m} (-1)^a\sum_{b=0}^n s(a,b)\cdot \left( \sum_{c=0}^i(-1)^c {i \choose c}(i-c+2)^{b} \right), \;\;\;\;\text{if $m$ is odd}. \label{eqn:f-v-odd}
\end{eqnarray}
Moreover, if $l = \min\{\#F~:~F\in E\}$ and $n - l + 2 \leq i \leq n$, then
\[
  f_i(\chi_H(k+1)) = \sum_{j=0}^i (-1)^j {i \choose j} (i-j+2)^n.
\]
\end{theorem}

\begin{proof}
By simple inclusion-exclusion, we obtain for all $0\leq i\leq n$ and every $0\leq m \leq \#E$
\begin{eqnarray*}
  f_i(\chi_H(k+1)) & = & f_i(C^n) + \sum_{a=1}^{\#E}(-1)^a \sum_{S\in {E \choose a}} f_i(\bigcap_{F\in S} P_F), \\
  f_i(\chi_H(k+1)) & \leq & f_i(C^n) + \sum_{a=1}^{m}(-1)^a \sum_{S\in {E \choose a}} f_i(\bigcap_{F\in S} P_F), \;\;\;\;\text{ if $m$ is even},\\
  f_i(\chi_H(k+1)) & \geq & f_i(C^n) + \sum_{a=1}^{m}(-1)^a \sum_{S\in {E \choose a}} f_i(\bigcap_{F\in S} P_F), \;\;\;\;\text{ if $m$ is odd}.
\end{eqnarray*}
By Remark~\ref{rem:intersection-of-P_F} we note that $\bigcap_{F\in S} P_F$ is a triangulation of some unit cube by the braid arrangement. In particular $f_i(\bigcap_{F\in S} P_F)$ depends only on the dimension of $\bigcap_{F\in S} P_F$, which allows us to gather terms. Thus, using the definition of $s(a,b)$ and 
Remark~\ref{rem:intersection-of-P_F} the first three formulas follow. The last identity follows from the fact that none of the complexes 
$P_F$ have faces of dimension $n-l+2$ or higher and thus $f_i(\chi_H(k+1)) = f_i(C^n)$ for $i\geq n-l+2$.
\end{proof}

As an application of the preceding theorem, we derive explicit upper and lower bounds for the $f$-vector.

\begin{corollary}
\label{cor:f-vector}
Let $n$ be a positive integer and $H=([n],E)$ by an $r$-uniform hypergraph with $r\geq 2$. Then for all $0\leq i \leq n$
\begin{eqnarray*}
\sum_{c=0}^i(-1)^c {i \choose c} \left( (i-c+2)^{n} - \#E\cdot (i-c+2)^{n-r+1} \right) 
& \leq \;
f_i(\chi_H(k+1)) \;
\leq &
\sum_{c=0}^i(-1)^c {i \choose c}(i-c+2)^{n}.
\end{eqnarray*}
\end{corollary}

\begin{proof}
The first inequality follows from (\ref{eqn:f-v-odd}) in Theorem~\ref{thm:f-vector} for $m=1$ using $s(0,n)=1$ 
and $s(0,b)=0$ for all other $b$ and the fact that for $r$-uniform hypergraphs $s(1,n-r+1)=\#E$ and $s(1,b)=0$ for $b\not=n-r+1$. 
The second inequality follows from (\ref{eqn:f-v-even}) in Theorem~\ref{thm:f-vector} for $m=0$ using $s(0,n)=1$ and $s(0,b)=0$ for all other $b$.
\end{proof}

Note that the upper bound in the above corollary holds for arbitrary hypergraphs, not just uniform ones.

\subsection{The $h$-vector of the coloring complex}

In this subsection we show that in general the $h$-vector of the coloring complex $\Delta$ may have negative entries. Since the $h$-vector of a 
partitionable simplicial complex is always non-negative (entry-wise) (see e.g., \cite[Proposition 2.3]{Stanley96}), this 
demonstrates that coloring complexes are not partitionable in general. We proceed by constructing an example.

\begin{example}\label{ex:not-pertitionable}
Consider the hypergraph $H$ on vertex set $[6]$ with edges $123$, $345$ and $156$. $\Box_H$ is a 4-dimensional complex, whence $h^4(L_{\Box_H}(k))=h(\Box_H)$. By inclusion-exclusion the Ehrhart function $L_{\Box_H}$ is given by
\begin{eqnarray*}
L_{\Box_H}(k) &=&
L_{P_{123}}(k) + L_{P_{345}}(k) + L_{P_{156}}(k) - L_{P_{12345}}(k) - L_{P_{13456}}(k) - L_{P_{12356}}(k) + L_{P_{123456}}(k).
\end{eqnarray*}
Applying Proposition~\ref{prop:f-of-cube}, we compute
\begin{eqnarray*}
h^4(L_{P_{123}}(k)) = h^4(L_{P_{345}}(k)) = h^4(L_{P_{156}}(k)) & = & (1,11,11,1,0,0) \\
h^4(L_{P_{12345}}(k)) = h^4(L_{P_{13456}}(k)) = h^4(L_{P_{12356}}(k)) & = & (1,-1,-1,1,0,0) \\
h^4(L_{P_{123456}}(k)) & = & (1,-3,3,-1,0,0)
\end{eqnarray*}
and so 
\[
h(\Box_H) = h^4(L_{\Box_H}(k)) = (1,33,39,-1,0,0).
\]
Now, $\Box_H$ is the double cone over the coloring complex $\Delta_H$. Removing the two cone points does not affect the $h$-vector, 
except that the last two entries are removed \cite[Exercise 7(a), p. 136]{Stanley96}. Thus $h(\Delta_H)=(1,33,39,-1)$, 
which shows in particular that the coloring complex of $H$ is not partitionable. 
Computational evidence suggests that the above construction may produce $r$-uniform hypergraphs with non-partitionable coloring complexes for all odd $r\geq 3$.
\end{example}

We summarize the results of this subsection in the following proposition.

\begin{proposition}
There exist uniform hypergraphs $H$ such that $h(\Delta_H)$ has negative entries and $\Delta_H$ is not partitionable.
\end{proposition}

Note that Example \ref{ex:not-pertitionable} provides yet another proof of Proposition \ref{prop:nonCM} since the entries of the $h$-vector 
of a Cohen-Macaulay complex are all non-negative, see e.g., \cite[Theorem 5.1.10]{BH-book}. Also, this implies that coloring complexes of hypergraphs do not in general have a convex ear decomposition, see e.g., \cite{HershSwartzCol}.

\section{The homotopy type of the coloring complex}
\label{sec:homtype}

In the following, we will use the notations introduced in the second part of Section \ref{subsect:simplicial}.  
The aim of this section is to investigate the homotopy type of the coloring complex of an arbitrary hypergraph. Whereas, classical 
coloring complexes of graphs are known to be homotopy equivalent to wedges of spheres of top dimension, 
it turns out that for hypergraph coloring complexes not that much can be said.
However, using the following special version of the Wedge Lemma from \cite{ZZ} we can at least provide a method of 
how to compute the homotopy type of the coloring complex of a graph.

\begin{wedgelemma}\cite[Wedge Lemma 6.1]{HRW-Koszul}
 Let $\U$ be a covering of a regular CW-complex $\Delta$ by closed subcomplexes $\Delta_1,\ldots,\Delta_l$. Let $P^{\U}$ be the 
intersection poset of $\U$. Assume that for all $p\in P^{\U}$ there is a point $c_p\in U_p$ such that for all $q>p$ the inclusion map 
$U_q\hookrightarrow U_p$ for $q>p$ is homotopic to a constant map which sends $U_q$ to $c_p$. Then 
$\Delta$ is homotopy equivalent to the wedge
\begin{equation*}
 \bigvee_{p\in P}\Delta(P_{<p})\ast U_p,
\end{equation*}
in which the wedge identifies the vertex $p$ in $\Delta(P_{<p})$ with the vertex $p$ in $\Delta(P_{<\hat{1}})$, 
where $\hat{1}$ is the top element of $P$ corresponding to the intersection $\bigcap_{i=1}^l \Delta_i$.
\end{wedgelemma}

We now explain how the above ``Wedge Lemma'' can be implied in our situation.

Given a hypergraph $H=([n],E)$ we have seen in Section \ref{subsect:combCol} that each edge $F\in E$ 
gives rise to a subcomplex $Q_F$ of $\Delta_H$, which was referred to as edge sphere previously. 
Moreover, by construction, it holds that $\Delta_H=\bigcup_{F\in E} Q_F$, which means that the family 
$\U^H=(Q_F)_{F\in E}$ is a covering of $\Delta_H$. 
To simplify notation, let $P^H$ denote the intersection poset $P^{\U^H}$ 
of this covering. In order to better understand the structure of $P^H$ we need to determine 
how the intersections $\bigcap_{F\in S}Q_F$ for $S\subseteq E$ look like. 
This is accomplished by the following lemma.

\begin{lemma}\label{lem:intersection}
Let $H=([n],E)$ be a hypergraph and let $S\subseteq E$.  
Let $H_S^{(1)},\ldots ,H_S^{(m)}$ denote the connected components of $H_S$. 
Then $\bigcap_{F\in S}Q_F$ is homeomorphic to a $d_S$-sphere, where $d_S=n-\sum_{i=1}^m n_S^{(i)}+m-2$.
Here, for $1\leq i\leq m$ we denote by $n_S^{(i)}$ the number of vertices in $H_S^{(i)}$.
\end{lemma}

\begin{proof}
Let $Q=\bigcap_{F\in S}Q_F$. In the following, we will characterize maximal faces of $Q$. 
Consider a maximal face $B=B_1|B_2|\cdots|B_r\in Q$. 
First note that vertices of $H_S$, belonging to the same connected component of $H_S$, have to lie in the same block. 
Since $B$ is a maximal face, this in particular means, that for each connected component $H_S^{(i)}$ of $H_S$, there exists a block 
$B_l$ of $B$ containing exactly the vertices of $H_S^{(i)}$. Again, by maximality of $B$, we know that the remaining blocks of 
$B$ have to be singletons. 
Altogether, we conclude that a facet of $Q$ consists of $m+(n-\sum_{i=1}^m n_S^{(i)})$ blocks and therefore $Q$ has to 
be of dimension $n-\sum_{i=1}^m n_S^{(i)}+m-2=d_S$. 
Moreover, by the same arguments as in Example \ref{Ex1} (ii) and \cite[Theorem 14]{Steingrimsson} it follows that 
$Q$ is simplicially isomorphic to the barycentric subdivision of the boundary of an $(n-\sum_{i=1}^m n_F^{(i)}+m-1)$-simplex 
and as such homeomorphic to an $(n-\sum_{i=1}^m n_F^{(i)}+m-2)$-sphere.
\end{proof}

As a direct consequence of the above lemma we get the following behavior of intersections of pairs of edge spheres.

\begin{remark}\label{rem:intersect}
Let $H=([n],E)$ be a hypergraph and let $F$, $F'\in E$ be two edges of $H$. 
By Lemma \ref{lem:intersection} their edge spheres, $Q_F$ and $Q_{F'}$, intersect in 
a sphere of dimension $n-\#F-\#F'$ and of dimension $n-\#(F\cup F')-1$, if $F$ and $F'$ are disjoint and share at least 
one common vertex, respectively. This means that in contrast to the situation for coloring complexes of ordinary graphs, the 
codimension of these intersections can become arbitrarily large. In particular, $Q_F\cap Q_F'=\emptyset$ if 
and only if $F\cup F'=[n]$ and $F\cap F'\neq \emptyset$.
\end{remark}

Now, consider two subsets $F_1$ and $F_2$ of the edge set of $H$ and 
let $p_{F_1}$ and $p_{F_2}\in P^H$ be the corresponding elements of the intersection poset $P^H$.
If $p_{F_1}< p_{F_2}$, then it directly follows from  
Lemma \ref{lem:intersection} that the inclusion map 
$U_{p_{F_2}}\hookrightarrow U_{p_{F_1}}$  
is just the inclusion of a $d_{F_2}$-sphere into a $d_{F_1}$-sphere and 
as such this map is homotopic to a constant map. 
Finally, the application of the ``Wedge Lemma'' yields the following proposition:

\begin{proposition}\label{prop:homotopy}
 Let $H=([n],E)$ be a hypergraph. Then the hypergraph coloring complex $\Delta_H$ is homotopy equivalent to
\[
 \bigvee_{p\in P^{H}}S^{d_p}\ast \Delta(P^{H}_{<p}),
\]
where $d_p$ is defined as in Lemma \ref{lem:intersection}.
\end{proposition}

It is clear from Proposition \ref{prop:homotopy} that the homotopy type of the coloring complex only depends 
on the order complexes of the lower intervals $P_{<p}$ in the intersection lattice $P^{h}$. The only 
thing we can generally say about those intervals is that the closed intervals $P_{\leq p}$ themselves are intersection lattices of 
coloring complexes of subhypergraphs of $H$ (having edges corresponding to the elements in the intersection $p$).

\subsection{Connectedness}

In this section, we are dealing with connectedness of hypergraph coloring complexes. Though coloring complexes of ordinary graphs 
are always connected, this property breaks down if one considers hypergraphs. But it is still possible to 
give a unique characterization of those hypergraphs which are connected. Moreover, we can construct hypergraphs whose 
coloring complexes have arbitrarily many connected components.

In order to give a necessary and sufficient criterion for the hypergraph coloring complex to be connected we 
need the following lemma which is a direct consequence of the discussion in Remark~\ref{rem:intersect}.

\begin{lemma}\label{lem:edges}
 Let $H=([n],E)$ be a hypergraph and let $F,F'\in E$ be two edges. Then $Q_F \cap Q_{F'} = \emptyset$ 
if and only if $F \cup F' = [n]$ and $F \cap F' \neq \emptyset$.
\end{lemma}

Finally, we obtain the following characterization of hypergraphs having a connected coloring complex.

\begin{proposition} 
\label{prop:connectedness}
Let $H=([n],E)$ be a hypergraph. Then the coloring complex 
$\Delta_H$ is connected if and only if for every pair of edges $F$, $F' \in  
E$ there is a sequence of edges $F = F_1, F_2, \ldots, F_r = F'$ such that 
$F_{i} \cup F_{i+1} \neq [n]$ or $F_{i} \cap F_{i+1} = \emptyset$ for $1\leq i\leq r-1$.
\end{proposition}

\begin{proof}
Given two edges $F$ and $F'$ and such a sequence between them, we have that $Q_{F_{i}} \cup Q_{F_{i+1}}$ is connected by Lemma \ref{lem:edges}. 
So $\bigcup_{i=1}^{r}Q_{F_{i}}$ is connected. 
Thus, any two edge spheres are contained in the same connected component of $\Delta_{H}$ which implies that $\Delta_{H}$ is connected.
This proves one direction.

Conversely, suppose $\Delta_{H}$ is connected. 
Let $F$, $F'\in E$ be any pair of edges. Since $\Delta_H$ is connected, there exists a sequence of edges 
$F = F_1, F_2, \ldots, F_r = F'$ such that $Q_{F_i}\cap Q_{F_{i+1}}\neq \emptyset$ for $1\leq i\leq r-1$. 
By Lemma \ref{lem:edges} the latter condition is equivalent to $F_i\cup F_{i+1}=[n]$ or $F_i\cap F_{i+1}=\emptyset$ for $1\leq i\leq r-1$.
This completes the proof.
\end{proof}

We close this section with an example showing that hypergraph coloring complexes can have arbitrarily many
connected components.

\begin{example}
Let $m\geq 2$ be an integer and $(a_1,\ldots,a_m)\in \NN^m$ be a vector of positive integers. 
We assume that $a_1\geq a_2\geq \cdots \geq a_m$. Let $a:=\max(3,a_1)$ and set
\begin{equation*}
 E_i:=\{[(m-i)a]\cup\{(m-i+1)a+1,\ldots,ma-1,ma\}\cup\{j\}~:~(m-i)a+1\leq j\leq (m-i)a+a_i\}
\end{equation*}
for $1\leq i\leq m$.
Let $H$ be the hypergraph on vertex set $[ma]$ whose edge set is $E=E_1\cup\cdots \cup E_m$. 
Consider two edges $F$, $F'\in E_i$. Since $a\geq 3$, it holds that $F\cup F'\neq [n]$. Hence, it follows from
Lemma \ref{lem:edges} that $Q_F\cap Q_{F'}\neq \emptyset$. In particular, $Q_F$ and $Q_{F'}$ lie in the same 
connected component of $\Delta_H$.  
On the other hand, if $F\in E_i$ and $F'\in E_j$ for $i\neq j$, then $F\cup F'=[ma]$ and $F\cap F'\neq \emptyset$. 
From Lemma \ref{lem:edges} we infer that $Q_F\cap Q_{F'}=\emptyset$. 
To summarize, we have shown that for any pair of edges $F$, $F'\in E$, their edge spheres $Q_F$ and $Q_{F'}$ 
belong to the same connected component of $\Delta_G$ if and only if there exists $1\leq i\leq m$ such that $F$, $F'\in E_i$. 
This means that the hypergraph coloring complex $\Delta_H$ of $H$ consists of $m$ connected 
components. Since $\#E_i=a_i$, for each $1\leq i\leq m$ there exists one component 
containing exactly $a_i$ edge spheres.
\end{example}

\subsection{Wedge of Spheres}

We have seen that hypergraph coloring complexes do not have many of the nice properties natural simplicial complexes often enjoy. One of the last properties that one might hope hypergraph coloring complexes to have is that if they are connected, they have the homotopy type of a wedge of spheres. Unfortunately, it turns out that, in general, even for uniform hypergraphs this property fails. 

%
In order to show this, we give a concrete example of a uniform hypergraph $H$, whose hypergraph coloring complex $\Delta_H$ is connected but which itself is not homotopy equivalent to a wedge of spheres. The underlying idea is to construct a torus out of edge spheres, as shown in Figure~\ref{fig:intuition}. The edges these spheres correspond to are shown in Figure~\ref{fig:definition}. For example, the sphere labled $A$ in Figure~\ref{fig:intuition} corresponds to the edge 12347 as shown in Figure~\ref{fig:definition}.

\begin{figure}[ht]
\begin{center}
\includegraphics[width=14cm]{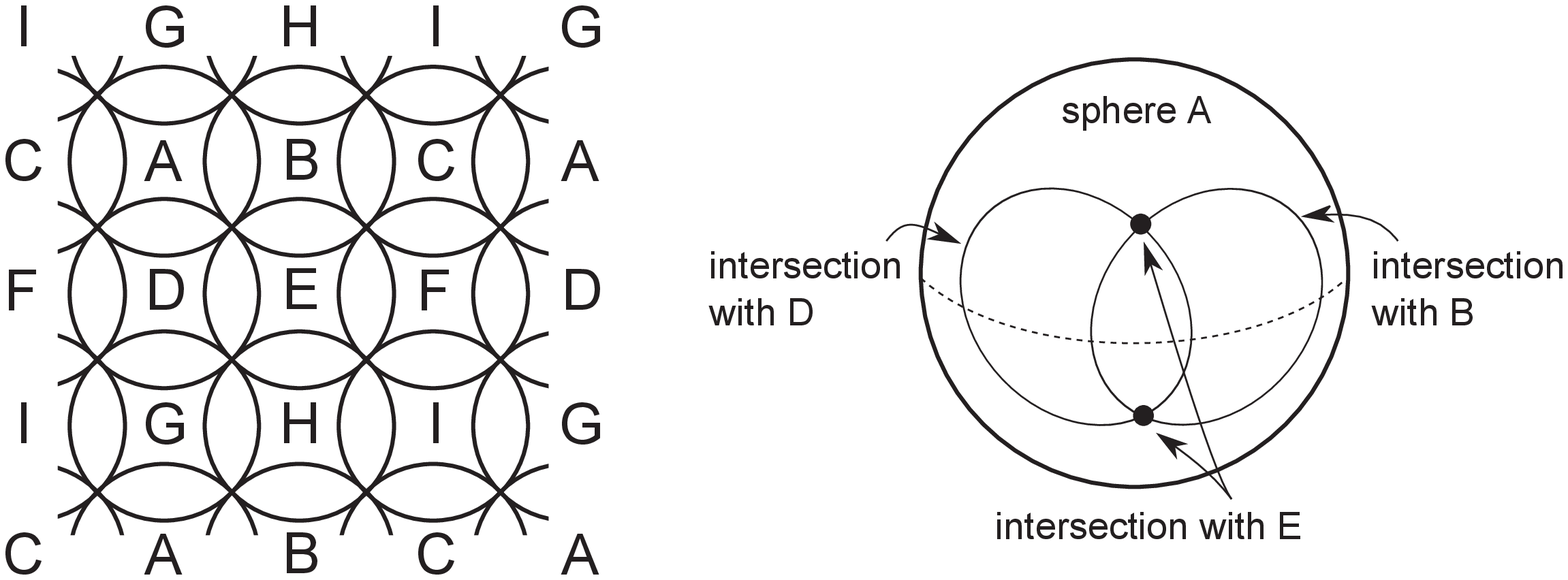}
\end{center}
\caption{\label{fig:intuition}The intuitive motivation for the construction of $\Delta_H$.}
\end{figure}

\begin{figure}[ht]
\begin{center}
\includegraphics[width=6cm]{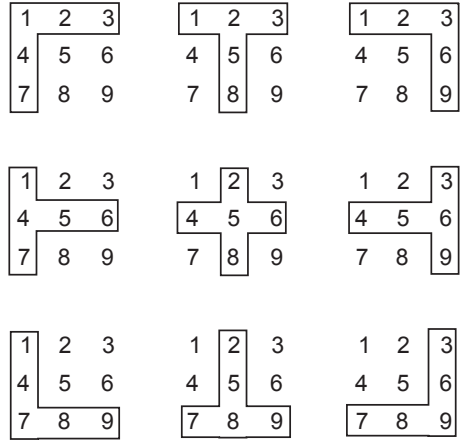}
\end{center}
\caption{\label{fig:definition}The edges of $H$ arranged to match the illustration in Figure~\ref{fig:intuition}.}
\end{figure}

\begin{example}
More precisely, consider the hypergraph $H=([9],E)$ with 
\[
  E=\{12347,12358,12369,14567,24568,34569,14789,25789,36789\}.
\]
It is easily seen that $\Delta_H$ is connected. However, as we will show $\Delta_H$ does not have the homotopy type of a wedge of spheres. For this aim, we show that the cup product defined on the cohomology groups of $\Delta_H$ is non-trivial. First, using the mathematical software system Sage \cite{Sage}, we computed the reduced cohomology groups of $\Delta_H$ over $\mathbb{Z}$ and obtained:
\begin{align*}
\widetilde{H}^0(\Delta_H;\mathbb{Z})&=0,\\
\widetilde{H}^1(\Delta_H;\mathbb{Z})&=\mathbb{Z}^2,\\
\widetilde{H}^2(\Delta_H;\mathbb{Z})&=\mathbb{Z}^{28},\\
\widetilde{H}^3(\Delta_H;\mathbb{Z})&=\mathbb{Z}^9.
\end{align*} 
In the next step, we implemented the computation of the cup product in cohomology in Sage. Taking two generators of $\widetilde{H}^1(\Delta_H;\mathbb{Z})$ and computing their cup product, we obtained a cohomology cycle in  $\widetilde{H}^3(\Delta_H;\mathbb{Z})$ that is not a coboundary and is thus not trivial in $\widetilde{H}^3(\Delta_H;\mathbb{Z})$. In particular, this shows that the cup product on the cohomology groups of $\Delta_H$ is not trivial and, hence, that $\Delta_H$ is not homotopy equivalent to a wedge of spheres. 
\end{example}

All edge spheres in this example are 3-dimensional. Any two edge spheres that are horizontally or vertically adjacent, for example $P_{12358}$ and $P_{24568}$, intersect in a 1-dimensional sphere. Any two edge spheres that are diagonally adjacent, for example $P_{12347}$ and $P_{24568}$, intersect in a 0-dimensional sphere. However, any three edge spheres meet all three columns or all three rows in Figure~\ref{fig:definition} have an empty intersection. This already suggests that the coloring complex $\Delta_H$ does indeed have the structure suggested by Figure~\ref{fig:intuition}.

We can summarize the results of this section in the following proposition.

\begin{proposition}
There exist uniform hypergraphs whose hypergraph coloring complexes are not homotopy equivalent to a wedge of spheres. 
\end{proposition}

\subsection*{Acknowledgments}
We would like to thank Volkmar Welker for helpful discussions and for bringing the Wedge Lemma to our attention. We are also grateful to two anonymous referees for comments, which helped to improve the contents of this paper. Finally, we would like to thank John H.\ Palmieri for answering several questions about the implementation of cohomology groups in Sage.

Felix Breuer was partially supported by grants HA 4383/1 and BR 4251/1-1 of the German Research Foundation (DFG). 
Aaron Dall was partially supported by the Spanish Ministry of Science and Innovation grant BES-2010-030080. Martina Kubitzke was supported by the Austrian Science Foundation (FWF) through grant Y463-N13.

\bibliographystyle{plain} 
\bibliography{citations}

\begin{thebibliography}{10}

\bibitem{oeis}
The on-line encyclopedia of integer sequences.
\newblock published electronically at http://oeis.org, 2011.

\bibitem{BabsonKozlov2006}
Eric Babson and Dmitry~N. Kozlov.
\newblock Complexes of graph homomorphisms.
\newblock {\em Israel J. Math.}, 152:285--312, 2006.

\bibitem{BabsonKozlov2007}
Eric Babson and Dmitry~N. Kozlov.
\newblock Proof of the {L}ov\'asz conjecture.
\newblock {\em Ann. of Math. (2)}, 165(3):965--1007, 2007.

\bibitem{BeckRobins07}
Matthias Beck and Sinai Robins.
\newblock {\em Computing the continuous discretely}.
\newblock Undergraduate Texts in Mathematics. Springer, New York, 2007.

\bibitem{BeckZaslavsky06a}
Matthias Beck and Thomas Zaslavsky.
\newblock Inside-out polytopes.
\newblock {\em Adv. Math.}, 205(1):134--162, 2006.

\bibitem{BeckZaslavsky06b}
Matthias Beck and Thomas Zaslavsky.
\newblock The number of nowhere-zero flows on graphs and signed graphs.
\newblock {\em J. Combin. Theory Ser. B}, 96(6):901--918, 2006.

\bibitem{Birkhoff1912}
George~D. Birkhoff.
\newblock A determinant formula for the number of ways of coloring a map.
\newblock {\em The Annals of Mathematics}, 14(1/4):42--46, 1912.

\bibitem{breuer-diss}
Felix Breuer.
\newblock {\em {Ham Sandwiches, Staircases and Counting Polynomials}}.
\newblock Dissertation, Freie Universit\"{a}t Berlin, 2009.

\bibitem{breuer-dall2}
Felix Breuer and Aaron Dall.
\newblock {Bounds on the Coefficients of Tension and Flow Polynomials}.
\newblock {\em Journal of Algebraic Combinatorics}, 2010.
\newblock Zur Ver\"offentlichung angenommen am 31.8.2010.

\bibitem{breuer-dall1}
Felix Breuer and Aaron Dall.
\newblock {Viewing counting polynomials as Hilbert functions via Ehrhart
  theory}.
\newblock In {\em 22nd International Conference on Formal Power Series and
  Algebraic Combinatorics (FPSAC 2010)}, pages 413--424. DMTCS, 2010.

\bibitem{breuer-sanyal}
Felix Breuer and Raman Sanyal.
\newblock {Ehrhart theory, Modular flow reciprocity, and the Tutte polynomial}.
\newblock {\em Mathematische Zeitschrift}, 2010.
\newblock Zur Ver\"offentlichung angenommen am 13.9.2010.

\bibitem{BH-book}
W.\ Bruns and J.\ Herzog.
\newblock {\em Cohen-{M}acaulay rings. {R}ev. ed.}, volume~39 of {\em Cambridge
  Studies in Advanced Mathematics}.
\newblock Cambridge University Press, 1998.

\bibitem{Hanlon}
Phil Hanlon.
\newblock A {H}odge decomposition interpretation for the coefficients of the
  chromatic polynomial.
\newblock {\em Proc. Am. Math. Soc.}, 136(11):3741--3749, 2008.

\bibitem{HershSwartzCol}
Patricia Hersh and Ed~Swartz.
\newblock Coloring complexes and arrangements.
\newblock {\em J. Algebraic Comb.}, 27(2):205--214, 2008.

\bibitem{HRW-Koszul}
J\"urgen Herzog, Vic Reiner, and Volkmar Welker.
\newblock The {K}oszul property in affine semigroup rings.
\newblock {\em Pacific J. Math}, 186:39--65, 1997.

\bibitem{Hultman}
Axel Hultman.
\newblock Link complexes of subspace arrangements.
\newblock {\em European Journal of Combinatorics}, 28(3):781--790, 2007.

\bibitem{Jonsson}
Jakob Jonsson.
\newblock The topology of the coloring complex.
\newblock {\em J. Algebraic Combin.}, 21(3):311--329, 2005.

\bibitem{LongRundell}
Jane~Holsapple Long and Sarah~Crown Rundell.
\newblock The {H}odge structure of the coloring complex of a hypergraph.
\newblock {\em Discrete Math.}, 311(20):2164--2173, 2011.

\bibitem{Lovasz}
L.~Lov\'{a}sz.
\newblock Kneser's conjecture, chromatic number, and homotopy.
\newblock {\em Journal of Combinatorial Theory, Series A}, 25(3):319 -- 324,
  1978.

\bibitem{Stanley96}
Richard~P. Stanley.
\newblock {\em Combinatorics and Commutative Algebra}, volume~41 of {\em
  Progress in Mathematics}.
\newblock Birkh{\"{a}}user, second edition edition, 1996.

\bibitem{Sage}
W.A. Stein et~al.
\newblock {\em {S}age {M}athematics {S}oftware ({V}ersion 4.7.2)}.
\newblock The Sage Development Team, 2011.
\newblock {\tt http://www.sagemath.org}.

\bibitem{Steingrimsson}
Einar Steingr\'{i}msson.
\newblock The coloring ideal and coloring complex of a graph.
\newblock {\em J. Algebraic Combin}, 14(1):73--84, 2001.

\bibitem{Wilf}
Herbert~S. Wilf.
\newblock Which polynomials are chromatic?
\newblock In {\em Colloquio {I}nternazionale sulle {T}eorie {C}ombinatorie
  ({R}oma, 1973), {T}omo {I}}, pages 247--256. Atti dei Convegni Lincei, No.
  17. Accad. Naz. Lincei, Rome, 1976.

\bibitem{XieLiu}
Li~Tong Xie and Gui~Zhen Liu.
\newblock Neighborhood complexes of graphs.
\newblock {\em Shandong Daxue Xuebao Ziran Kexue Ban}, 28(1):40--44, 1993.

\bibitem{Ziegler95}
G{\"{u}}nter~M. Ziegler.
\newblock {\em Lectures on Polytopes}, volume 152 of {\em Graduate Texts in
  Mathematics}.
\newblock Springer-Verlag, 1995.

\bibitem{ZZ}
G\"unter~M. Ziegler and Rade~T. \v{Z}ivaljevi\'c.
\newblock Homotopy types of subspace arrangements via diagrams of spaces.
\newblock {\em Mathematische Annalen}, 295:527--548, 1993.

\end{thebibliography}

\end{document}